\pdfoutput=1

\documentclass[12pt]{article}
\usepackage{etex}

\usepackage[top=1.2in, bottom=1.2in, left=1in, right=1in]{geometry}

\usepackage{amsmath}
\usepackage{amssymb}
\usepackage{amsthm}
\usepackage{array}
\usepackage{bbm}
\usepackage{bm}
\usepackage{cancel}
\usepackage{cmap}
\usepackage{csquotes}
\usepackage{enumerate}
\usepackage{enumitem}%
\usepackage{esint}%
\usepackage{fancyhdr}
\usepackage{listings}
\usepackage{longtable}
\usepackage{makeidx}
\usepackage{mathdots}%
\usepackage{mathtools}
\usepackage{mathrsfs}
\usepackage{stackrel}
\usepackage{stmaryrd}%
\usepackage{tabularx}
\usepackage{ctable}
\usepackage{titlesec}
\usepackage{titletoc}
\usepackage{url}
\usepackage{verbatim}
\usepackage{wasysym}
\usepackage{wrapfig}
\usepackage{yhmath}
\usepackage{authblk}
\usepackage[T1]{fontenc}
\usepackage{lmodern}

\usepackage{bookmark,hyperref}

\usepackage[backend=biber,style=alphabetic,maxbibnames=99]{biblatex}

\usepackage[refpage]{nomencl}

\makenomenclature

\makeindex

\usepackage{thmtools,thm-restate}

\newtheoremstyle{norm}
{12pt}
{12pt}
{}
{}
{\bf}
{:}
{.5em}
{}

\newtheorem{thm}{Theorem}[section]
\newtheorem*{thm*}{Theorem}

\newtheorem*{clm*}{Claim}
\newtheorem{conj}[thm]{Conjecture}
\newtheorem*{conj*}{Conjecture}
\newtheorem{cor}[thm]{Corollary}
\newtheorem{lem}[thm]{Lemma}
\newtheorem*{lem*}{Lemma}

\theoremstyle{norm}
\newtheorem*{prb*}{Problem}

\newtheorem*{ax*}{Axiom}

\newtheorem*{df*}{Definition}

\newtheorem*{ex*}{Example}

\newtheorem*{pos*}{Postulate}

\newtheorem*{pr*}{Proposition}

\newtheorem*{qu*}{Question}

\newtheorem*{rem*}{Remark}

\newcommand{\E}[0]{\mathbb{E}}

\newcommand{\cM}[0]{\mathcal{M}}

\newcommand{\N}[0]{\mathbb{N}}

\newcommand{\be}[0]{\beta}
\newcommand{\ga}[0]{\gamma}
\newcommand{\Ga}[0]{\Gamma}
\newcommand{\de}[0]{\delta}
\newcommand{\De}[0]{\Delta}
\newcommand{\ep}[0]{\varepsilon}

\newcommand{\rh}[0]{\rho}

\newcommand{\om}[0]{\omega}
\newcommand{\Om}[0]{\Omega}
\newcommand{\si}[0]{\sigma}

\newcommand{\subeq}[0]{\subseteq}

\newcommand{\bs}[0]{\backslash}
\newcommand{\iy}[0]{\infty}

\newcommand{\rc}[1]{\frac{1}{#1}}

\newcommand{\fc}[2]{\frac{#1}{#2}}

\newcommand{\pa}[1]{\left( {#1} \right)}

\newcommand{\set}[2]{\left\{{#1}:{#2}\right\}}

\newcommand{\pull}[9]{
#1\ar@/_/[ddr]_{#2} \ar@{.>}[rd]^{#3} \ar@/^/[rrd]^{#4} & &\\
& #5\ar[r]^{#6}\ar[d]^{#8} &#7\ar[d]^{#9} \\}

\newcommand{\cmp}[9]{
\xymatrix{
#1 \ar[r]^{#4}{#5} \ar@/_2pc/[rr]^{#8}_{#9} & #2 \ar[r]^{#6}_{#7} & #3
}
}

\newcommand{\ha}[1]{\ar@{^(->}[#1]}
\newcommand{\ls}[1]{\ar@{-}[#1]}
\newcommand{\sj}[1]{\ar@{->>}[#1]}
\newcommand{\aq}[1]{\ar@{=}[#1]}
\newcommand{\acir}[1]{\ar@{}[#1]|-{\textstyle{\circlearrowright}}}
\newcommand{\acil}[1]{\ar@{}[#1]|-{\textstyle{\circlearrowleft}}}
\newcommand{\ard}[1]{\ar@{.>}[#1]}
\newcommand{\mt}[1]{\ar@{|->}[#1]}
\newcommand{\inm}[1]{\ar@{}[#1]|-{\in}}
\newcommand{\inr}{\ar@{}[d]|-{\rotatebox[origin=c]{-90}{$\in$}}}
\newcommand{\inl}{\ar@{}[u]|-{\rotatebox[origin=c]{90}{$\in$}}}

\newcommand{\beq}[1]{\begin{equation}\llabel{#1}}
\newcommand{\eeq}[0]{\end{equation}}
\newcommand{\bal}[0]{\begin{align*}}
\newcommand{\eal}[0]{\end{align*}}%
\newcommand{\ban}[0]{\begin{align}}
\newcommand{\ean}[0]{\end{align}}

\newcommand{\fixme}[1]{{\color{red}#1}}
\newcommand{\llabel}[1]{\label{#1}\text{\fixme{\tiny#1}}}

\newcommand{\arxiv}[1]{\url{http://www.arxiv.org/abs/#1}}

\allowdisplaybreaks[2]

\DeclareFontFamily{U}{wncy}{}
    \DeclareFontShape{U}{wncy}{m}{n}{<->wncyr10}{}
    \DeclareSymbolFont{mcy}{U}{wncy}{m}{n}
    \DeclareMathSymbol{\Sh}{\mathord}{mcy}{"58} 

\newcommand{\dc}{d_{\mathrm{C}}}
\newcommand{\dtv}{d_{\mathrm{TV}}}
\newcommand{\tmix}{t_{\mathrm{mix}}}

\renewcommand{\subset}{\subseteq}

\title{Sampling List Packings}
\author[1]{Evan Camrud}
\author[2]{Ewan Davies\thanks{Supported in part by NSF grant CCF-2309707.}}
\author[2]{Alex Karduna}
\author[3]{Holden Lee}

\affil[1]{Department of Mathematics, Colorado State University}
\affil[1]{{\texttt{Evan.Camrud@colostate.edu}}}

\affil[2]{Department of Computer Science, Colorado State University}
\affil[2]{{\texttt{\{Ewan.Davies,Alex.Karduna\}@colostate.edu}}}

\affil[3]{Department of Applied Mathematics and Statistics, Johns Hopkins University}
\affil[3]{{\texttt{hlee283@jhu.edu}}}
    
\date{February 5, 2024}

\begin{document}

\maketitle

\begin{abstract}
    We study the problem of approximately counting the number of list packings of a graph. 
    The analogous problem for usual vertex coloring and list coloring has attracted a lot of attention. For list packing the setup is similar but we seek a full decomposition of the lists of colors into pairwise-disjoint proper list colorings. In particular, the existence of a list packing implies the existence of a list coloring. Recent works on list packing have focused on existence or extremal results of on the number of list packings, but here we turn to the algorithmic aspects of counting. 

    In graphs of maximum degree $\Delta$ and when the number of colors is at least $\Omega(\Delta^2)$, we give an FPRAS based on rapid mixing of a natural Markov chain (the Glauber dynamics) which we analyze with the path coupling technique. 
    Some motivation for our work is the investigation of an atypical spin system, one where the number of spins for each vertex is much larger than the graph degree.
\end{abstract}

\section{Introduction}

Recall that the classic graph coloring problem is to determine, for a graph $G=(V,E)$ and number of colors $q$, whether there is a coloring $f:V\to[q]$ such that $f$ is proper in the sense that $f(u)\ne f(v)$ for every edge $uv\in E$. 
List coloring emerged in the late 20th century as an adversarial version of this problem~\cite{Viz76,ERT80}.
To define list coloring, we take a list size $q\in\mathbb{N}$ and assign to each $u\in V$ a list of allowed colors $L(u)$ of size $q$. One can think of the list assignment $L$ as supplied by an adversary. 
If, under any choices of such an assignment of lists, the graph $G$ admits a proper coloring $f$ such that $f(u)\in L(u)$ for every vertex $u$, then we say that $G$ is $q$-list colorable (also known as $q$-choosable in some works).
In some ways list coloring and classical graph coloring are similar, e.g.,\ the complete graph on $n$ vertices requires $q\ge n$ for both problems. 
But all bipartite graphs can be colored with two colors, while $K_{n,n}$ requires lists of length $\Om(\log n)$ in the list coloring setting~\cite{ERT80}.

List coloring also arises naturally in various other ways. Notably, if we take a classic graph coloring problem and pre-color some vertices then we can express the problem of completing the coloring through list coloring. 
We set $L(u)$ to be the subset of the colors not used by any pre-colored vertex in the neighborhood of $u$ (though these lists may not have equal sizes).
The idea of list coloring has been extended in several interesting directions, and in this work we study a very recent and structured variant known as list packing~\cite{CCDK24}.
The setup is the same as for list coloring with a fixed list size $q$ but, instead of seeking just one proper coloring where each vertex is colored from its list, we seek $q$ pairwise-disjoint proper colorings from the lists.
We consider two colorings $f$ and $f'$ of a graph disjoint if, for all vertices $u$ we have $f(u)\ne f'(u)$.
We call this collection of $q$ pairwise-disjoint proper list colorings a \emph{list packing}. 
Given a graph $G$ and list size $q$, if a list packing can be found for any lists of size $q$ chosen by an adversary then we say that $G$ is $q$-list packable.

One motivation for list packing in~\cite{CCDK24} is to challenge the state-of-the-art in list coloring. 
For example, a folklore result states that a bipartite graph of maximum degree $\De$ is $q$-list colorable for $q\ge (1+o(1))\De/\log\De$ (see also~\cite{ACK21} and a recent improvement of the leading constant due to Bradshaw~\cite{Bra22}). 
Amongst other foundational results on list packing, the folklore result was matched for the significantly more structured notion of list packing in~\cite{CCDK24}.
A notable difference between known results for list packing and list coloring is that a general graph of maximum degree $\De$ is $q$-list packable for $q\ge 2\De$ (improved to $q\ge 2\De-2$ for $\De\ge 4$ in~\cite{CCDK23}); whilst such $G$ are $q$-list colorable for $q\ge \De+1$ (and even $q=\De$ if $\De\ge 3$ and $G$ does not contain a clique on $\De+1$ vertices~\cite{ERT80}).
Despite the gap in known results, there is scant evidence that more than $\De+1$ colors are ever required for list packing~\cite{CCDK23,CCDK24}. 

Early work on list packing has focused on the problem of existence, though many arguments for existence also provide efficient constructions of list packings. 
An extremal perspective on counting list packings was recently investigated by Kaul and Mudrock~\cite{KM24}, and related problems where one seeks many list colorings or ``flexible'' list colorings are studied in~\cite{KMMP22,CCZ23}.
In this work we turn to the study of approximately counting the number of $L$-packings of a graph $G$ with a fixed $q$-list assignment $L$.
This is a natural question by analogy with the same questions for graph coloring and list coloring: is there an efficient procedure that, given a graph $G$ and a number of colors $q$, approximates the number of proper $q$-colorings of $G$ (e.g.,\ to within a factor 2)?
This question is well-studied in the field of approximate counting and sampling, and is an important test-bed for algorithmic techniques. 
A longstanding open question is the existence of such an approximate counting algorithm that works for all $q\ge \De+1$ on graphs of maximum degree $\De$. 
An influential collection of results using various techniques requires conditions such as $q\ge e\De+1$~\cite{BDPR21}, $q\ge 2\De$~\cite{Jer95,LSS19}, $q\ge 11\De/6$~\cite{Vig99}, and even $q\ge(11/6-\ep)\De$ for some small $\ep>0$~\cite{CDM+19}.
Another branch of research on the algorithmic aspects of counting graph colorings seeks to sample perfectly uniformly from the set of proper $q$-colorings of a graph. 
The pioneering result is due to Huber~\cite{Hub98}, with improvements and new techniques supplied in a number of later works~\cite{LY13,LSS19,BC20a,JSS21}. 
Interestingly, one application of such perfect samplers is to design approximate counting algorithms that can be faster than analogous approaches which use approximate samplers.

Motivated by simple and poweful ideas such as the Markov chain Monte Carlo approach to counting colorings due to Jerrum~\cite{Jer95} (see also~\cite{SS97}), we seek similar results for list packing. 
The list packing problem, however, presents novel difficulties. 
Observe that finding a list coloring gets strictly easier as the list size $q$ grows. 
Supposing that one has a technique that works with lists of size $q_0$, then given larger lists one can take arbitrary subsets of the lists of size $q_0$ and apply the technique in a black-box fashion. 
In contrast, for list packing with larger $q$ we are required to find ever more list colorings which must also be pairwise-disjoint. Given lists of size $q>q_0$, it is not at all clear how to extend an arbitrary collection of $q_0$ pairwise-disjoint list colorings to a full list packing. 
The methods of~\cite{CCDK23,CCDK24} seem to show that list packing does get easier with larger $q$, albeit for less straightforward and general reasons.
We confirm this principle in the setting of approximate counting by giving an algorithm that approximately counts list packings which works for $q=q(\De)$ large enough in graphs of maximum degree $\De$. 

Approximate counting of combinatorial objects such as list colorings, independent sets, and matchings corresponds to approximating the so-called partition function of a spin system from statistical physics. 
We are interested in the study of counting and sampling list packings as it presents an unusual type of spin system. Typically, the number of spins available for each vertex in a spin system is small. 
In the Ising model of magnetism vertices take a spin from $\{+,-\}$, and in the (antiferromagnetic) Potts model associated with proper $q$-colorings, the spins are the $q$ colors. 
Parameter ranges frequently studied for this model on graphs of maximum degree $\De$ include the case when $q$ is close to $\De+1$.
The natural spin system associated with $q$-list packings, however, has $q!$ spins for each vertex. Since a vertex $u$ has a list of $q$ colors which we must decompose into $q$ choices, one for each list coloring in the packing, the spins for $u$ naturally correspond to permutations of the list $L(u)$. One of our contributions is to study a somewhat natural combinatorial problem which involves a spin system on bounded degree graphs with many more spins than commonly-studied examples. 
We hope that further insights into algorithmic techniques for approximate counting can be gained by studying an unusual spin system.

\section{Results}

Our main result is the existence of an approximate counting algorithm for list packings. 
We define an $\ep$-approximation of a real number $x$ as a real number $y$ satisfying $e^{-\ep}\le x/y\le e^\ep$.
We use the standard notion of a \emph{fully polynomial-time randomized approximation scheme} (FPRAS) for a counting problem. 
This is a randomized algorithm that, given $\ep>0$, yields an $\ep$-approximation of the true answer in time polynomial in the input size and $1/\ep$.

Before we state the result, we need some more notation for list packings focusing on a specific list assignment $L$.
Given a graph $G$ we call an assignment $L:V(G)\to2^{\mathbb{N}}$ of lists of colors to the vertices of $G$ such that $|L(u)|=q$ for all vertices $u$ a $q$-list assignment of $G$. 
A proper coloring $f$ of $G$ such that $f(u)\in L(u)$ for all vertices $u$ is called an $L$-coloring, or a list coloring if the lists $L$ are understood from context.
Given a graph $G$ and a $q$-list assignment $L$, we call a collection of $q$ pairwise-disjoint $L$-colorings of $G$ an $L$-packing.

\begin{thm}\label{t:main-count}
    There is an absolute constant $C$ such that the following holds. 
    For any $\Delta\ge 1$ and $q\ge C\Delta^2$, let $G$ be a graph of maximum degree $\De$. Then for any $q$-list assignment $L$ of $G$, there is an FPRAS for the number of $L$-packings of $G$. 
\end{thm}

We prove this result by analyzing the so-called ``heat-bath Glauber dynamics'' for list packing. We define and analyze this Markov chain later, noting here that we prove Theorem~\ref{t:main-count} by establishing rapid mixing of this chain and hence an approximate sampling algorithm for a uniform $L$-packing under the same conditions. 
Throughout, we assume that $\De$ and $q$ are fixed constants and do not analyze the case where they are allowed to depend on the number $n$ of vertices of $G$. We do not attempt to optimize the dependence of the running time on $q$ or $\De$.

A natural combinatorial problem on matchings in balanced bipartite graphs of large minimum degree emerges during the proof of Theorem~\ref{t:main-count}, leading to a probabilistic result stated below that may be of independent interest. 
Given a set $A$, we write $\mathcal{U}_A$ %
for the uniform distribution on $A$. When a fixed $q$ is clear from context, let $\dc$ be the Cayley metric on the symmetric group $S_q$. That is, $\dc(\rh,\rh')$ is the minimum number of transpositions which one must compose with $\rh$ to turn it into $\rh'$.
This is merely graph distance on the Cayley graph of $S_q$ generated by the transpositions.
We associate perfect matchings in a balanced bipartite graph $H=([q]\sqcup [q], E)$ with permutations $\rh\in S_q$ where $(i,\rh(i))\in E$ for each $i$.  

Given two random variables $X,Y$ defined on the discrete probability spaces $(\Omega_X, p_X)$ and $(\Omega_Y, p_Y)$, a coupling of $X,Y$ is a random variable $\ga=(X',Y')$ defined on a probability space $(\Omega_X\times \Omega_Y, p_\ga)$ such that $X$ and $X'$ have identical distributions and $Y$ and $Y'$ have identical distributions. 

Given a Markov chain $Z_t$ with state space $\Omega$ and transition matrix $P$, a coupling of $Z_t$ is a Markov chain $(X_t,Y_t)$ with state space $\Omega\times\Omega$ and transition matrix $\widehat P$ satisfying
\begin{align*}
    \sum_{y'\in\Omega}\widehat P((x,y),(x',y')) &= P(x,x'),\\
    \sum_{x'\in\Omega}\widehat P((x,y),(x',y')) &= P(y,y').
\end{align*}
That is, each coordinate of the coupling is a faithful copy of $Z_t$, though the transitions of the coordinates are not necessarily independent.

\begin{restatable}{lem}{lcoupledist}\label{l:couple-dist}
    There are constants $C_1$, $C_2$ such that the following hold. Suppose that $q\ge C_1\De$, and
    let $H=(V,E)$ be a bipartite graph in which $|V|=2q$ and the bipartition is balanced. Suppose that $H$ has minimum degree at least $q-\De$. 
    Let $e\in E$ be an edge of $H$, let $\mathcal{L}$ be the set of perfect matchings of $H$ containing $e$, and let $\mathcal{R}$ the set of perfect matchings of $H$ not containing $e$.
    Then there is a coupling $\ga$ of $\mathcal{U}_{\mathcal{R}}$ and $\mathcal{U}_{\mathcal{L}\cup\mathcal{R}}$ such that
    \[
\E_{(\rh,\rh')\sim \ga} [\dc(\rh,\rh')] \le \fc{C_2\De}{q}.
    \]
\end{restatable}

\noindent
This lemma is the main bottleneck for improving the dependence of $q$ on $\De$ in our main theorem. In particular, removing the factor of $\De$ would also improve Theorem~\ref{t:main-count} by a factor of $\De$.

\section{Technical overview}\label{s:overview}

We prove Theorem~\ref{t:main-count} by the standard Markov chain Monte Carlo approach and the path coupling technique of Bubley and Dyer~\cite{BD97}.
We study an ergodic Markov chain---the heat-bath Glauber dynamics---whose stationary distribution is uniform on the list packings of a graph, and use path coupling to show rapid mixing. 
Then a well-known and generic reduction from counting to sampling yields Theorem~\ref{t:main-count}.
Path coupling reduces the potentially challenging task of proving rapid mixing of a Markov chain to designing a coupling on adjacent states according to some graph $\Ga$ on the state space $\Om$. 
We largely follow the notation of~\cite{DG99} and consider
\[\cM:=\{P^t\,p_0\}_{t=0}^\infty,\]
an ergodic Markov chain on state space $\Omega$ with initial distribution $p_0$ and transition operator $P$. We denote the (unique) stationary distribution by $\pi$, and denote by $p_t:=P^t p_0$ the distribution of the state of the chain $\cM$ after $t$ steps.
We use the standard notion of mixing time of Markov chains given by
\[ \tmix(\ep) := \max_{p_0}\min\left\{t\ge 0 : \dtv(p_t,\pi)\le \ep\right\}, \]
where $\dtv$ is total variation distance.

Typically, we write $\om$ and $\om'$ for states of the chain before a transition and $\si:=P\om$, $\si':=P\om'$ for the (random) states after one step of the chain from $\om$ and $\om'$ respectively.

\begin{thm}[Bubley and Dyer~\cite{BD97}, see also~\cite{DG99}]\label{t:pathcoupling}
    Let $\Om$ be the state space of a Markov chain $\cM$ and let $\Ga$ be a weighted, directed graph on vertex set $\Om$ with edge weights in $\N$. 
    Let $\de$ be the quasi-metric\footnote{that is, a function which satisfies the conditions of a metric except symmetry} on $\Om$ given by taking shortest paths in $\Ga$, and suppose that $\de(\om,\om')\le D<\infty$ for all $\om,\om'\in\Om$.
    
    For each $(\om,\om')\in E(\Ga)$, suppose that we have a coupling $\ga$, of the random variables $\si$ (with distribution $P\om$) and $\si'$ (with distribution $P\om'$) for which $\E_{(\si,\si')\sim\ga} [\de(\si,\si')] \le \beta\de(\om,\om')$.
    Then if $\beta<1$, we have $\tmix(\ep) \le \log(D/\ep)/(1-\beta)$.
\end{thm}

The theorem may seem rather abstract, so we briefly discuss a well-known application to list coloring as a warm-up to the main argument for list packing. 
Let $G$ be a graph of maximum degree $\De$ and let $L$ be a $q$-list assignment of $G$. Let $\Om$ be the set of $L$-colorings of $G$, and let $\cM$ be the heat-bath Glauber dynamics on $\Om$, defined as follows. 
Note that states $\om\in\Om$ are colorings and hence functions $V(G)\to\N$.
A transition of $\cM$ from a state $\om$ is performed by choosing a vertex $u\in V(G)$ uniformly at random, sampling a color $c$ in $L(u)\setminus \om(N(u))$, and moving to the state $\si$ with $\si(u)=c$ and $\si(v)=\om(v)$ for all $v\ne u$. That is, the transition operator $P$ is defined via
\begin{enumerate}
    \item sampling a vertex $u\sim \mathcal{U}_{V(G)}$,
    \item sampling a color $c\sim\mathcal{U}_{L(u)\setminus\omega(N(u))}$,
    \item defining $\sigma$ by $\sigma(v):=\begin{cases}
c & v=u\\
\omega(v) & v\neq u
\end{cases}$, and moving to the state $\sigma$.
\end{enumerate}
Note that $L(u)\setminus \om(N(u))$ is the set of available colors for $u$. 
These colors are in the list $L(u)$ but are not used by the coloring $\om$ on the neighbors of $u$ so setting $\omega(u)$ to one of them yields a valid list coloring.
This chain is reversible, and when $q\ge\De+2$ it is ergodic~\cite[Exercises 4.1]{Jer03} with stationary distribution uniform on $\Om$.

Let $\Ga$ be the directed graph on $\Om$ where $(\om,\om')$ is an edge if and only if the colorings $\om,\om'$ differ at exactly one vertex and let all edge weights be 1.
From the proof of ergodicity~\cite{Jer03} it follows that we can take $D=(\Delta+1)n$ in Theorem~\ref{t:pathcoupling} and define couplings as follows. 
Let $(\om,\om')\in E(\Ga)$ and suppose that the colorings $\om$ and $\om'$ differ at the vertex $v^*$. 
Sample $u\in V(G)$ uniformly at random and update the color of $u$ in both chains (intuitively, this decision helps the chains coalesce). That is, the distribution $\ga$ of $(\si,\si')$ is defined by the Markov transition $(\om,\om')\overset{P_\gamma}{\mapsto}(\si,\si')$ itself defined by
\begin{enumerate}
\item sampling a vertex $u\sim \mathcal{U}_{V(G)}$,
\item sampling a pair of available colors $(a,b)$ from a distribution $\gamma_c$ such that
\begin{equation}
    \gamma_c:=\underset{\substack{\gamma'\text{ is a coupling of }\\ \mathcal{U}_{L(u)\setminus\omega(N(u))}\text{ and }\mathcal{U}_{L(u)\setminus\omega'(N(u))}}}{\arg\max}\Pr_{(a,b)\sim\gamma'}(a=b)
\end{equation}
\item defining $\si$ and $\si'$ by updating the color of $u$ in each to $a$ and $b$ respectively:
\begin{align*}
    \si(v) &:=\begin{cases}
a & v=u\\
\omega(v) & v\neq u\end{cases},&
\si'(v) &:=\begin{cases}
b & v=u\\
\omega(v) & v\neq u\end{cases},
\end{align*}
and moving to the state $(\si,\si')$.
\end{enumerate}
Observe that $\gamma_c$ is defined as the coupling on the uniform distributions of available colors, $\mathcal{U}_{L(u)\setminus \om(N(u))}$ and $\mathcal{U}_{L(u)\setminus \om'(N(u))}$, which maximizes the probability the colors are the same. The important property of this definition in terms of applying Theorem~\ref{t:pathcoupling} is that this coupling minimizes the expectation of the discrete metric on the sample.

If $u\in N(v^*)$ the sets of available colors for $u$ can be different making $\gamma_c$ nontrivial. In this case, let $C=\om(N(u)\setminus\{v^*\})$ and note that the two sets of available colors are $A=(L(u)\setminus C)\setminus\{\om(v^*)\}$ and $B=(L(u)\setminus C)\setminus\{\om'(v^*)\}$. We wish to couple $\mathcal{U}_A$ and $\mathcal{U}_B$ such that the probability of choosing the same color is maximized, and the best general coupling is not too hard to find. 

\begin{lem}[{See e.g.,~\cite[Lemma 4.10]{Jer03}}]\label{l:couple-uni-elts}
Let $U$ be a finite set and $A,B\subset U$. Then there is a coupling $\ga_c$ of $\mathcal{U}_A$ and $\mathcal{U}_B$ such that
\[\Pr_{(a,b)\sim\ga_c}(a=b) = \frac{|A\cap B|}{\max\{|A|,|B|\}},\]
where $(a,b)$ is a random element of $A\times B$ chosen according to the coupling $\ga_c$.
\end{lem}

Let $m=\max\{|A|,|B|\}$. 
We have $A \cap B = (L(u)\setminus C)\setminus\{\om(v^*),\om'(v^*)\}$. 
Checking the four cases according to whether each of $\om(v^*)$ and $\om'(v^*)$ are in $L(u)\setminus C$, we observe that 
$|A\cap B|\ge m-1$. By Lemma~\ref{l:couple-uni-elts}, we can ensure that the two chains choose the same color with probability at least $1-1/m$. 
Considering the definitions of $A$ and $B$, we also have $m\ge q-\De$.
Returning to the analysis of the coupling $\gamma$ described above, we have
\[ \E_{(\si,\si')\sim\gamma}[\de(\si,\si')] \le 1 + \frac{1}{n}\left(-1 + \frac{\De}{q-\De}\right). \]
This comes from the facts that $\de(\om,\om')=1$ by assumption, the probability $1/n$ that we successfully reduce the distance by $1$ in an update to $v^*$, and the probability of at most $\De/n$ that we choose to update a neighbor of $v^*$, and in this case fail to choose the same color in the coupling of the color choice in each chain given by Lemma~\ref{l:couple-uni-elts} (which occurs with probability at most $1/(q-\De)$ given that we update a neighbor of $v$).
Set $\be$ to the right-hand side above and solve for $\beta<1$ to obtain $q>2\De$.

Well-known works that first studied this technique~\cite{Jer95,BD97,Jer03} give a similar proof, though Jerrum~\cite{Jer95} manually constructed a coupling of the Markov chain and did not appeal to path coupling. 
With path coupling, it is slightly easier to study a variant of the Markov chain known as the ``Metropolis Glauber dynamics'' where the transition is defined slightly differently, though the same lower bound on $q$ is required in the argument for this chain.

Our argument for list packing follows the same outline as above using the heat-bath dynamics for list packings, but it is much harder to construct the coupling. In particular, we need an analogue of Lemma~\ref{l:couple-uni-elts} for the much more intricate combinatorial setting of list packings. 
We describe this in the next subsection.

\subsection{Coupling perfect matchings}

In the list packing setting, each ``spin'' is a permutation of a list rather than an element of the list. 
The central problem one faces when adapting the above sketch to list packing is the issue of coupling the choice of available permutations in the case that we are updating the spin of $u$ in two copies of the Glauber dynamics which differ at a neighbor $v\in N(u)$. 
It turns out (see e.g.~\cite{CCDK24} and earlier works such as~\cite{Mac21}) that there is an auxiliary bipartite graph in which available permutations for $u$ correspond to perfect matchings. 

Given a graph $G$ with $q$-list assignment $L$, let $\Om$ be the set of $L$-packings of $G$, and let $\om\in\Om$. 
Let $f_1,\dotsc,f_q$ be the $L$-colorings represented by $\om$. 
To be explicit, suppose that for each vertex $v$ of $G$ we write $L(v)$ in ascending order as $c_{1,v},\dotsc,c_{q,v}$. 
Then we identify $V(G)$ with $[n]$ and consider $\Omega\subset S_q^n$ such that with $v\in [n]$ we define $f_1,\dotsc,f_q:V(G)\to\mathbb{N}$ by $f_i(v) = c_{\om_v(i),v}$. 
This notation has the advantage that many of the dependencies are explicit, but to avoid a proliferation of subscripts we omit them where it is possible to fix some context in advance.

For a vertex $u\in V(G)$, we construct the \emph{availability graph} $H_u=H_u(G,L,\om)$ as follows. We consider the vertex set of $H_u$ as $[q]\sqcup[q]$, the disjoint union of two copies of the set $[q]$. The left copy consists of ``packing indices'' and the right copy consists of ``list indices''. For clarity, the edges of $H_u$ are considered oriented from left to right so that $(i,j)$ joins packing index $i$ to list index $j$. 
Suppose that $L(u)=\{c_1,\dotsc,c_q\}$ is supplied in some fixed order. Then in $H_u$ we include each edge $(i,j)\in[q]^2$ such that $c_j$ is an available color for $u$ in the coloring $f_i$. 
That is, $(i,j)$ is an edge of $H_u$ if and only if $c_j\notin f_i(N(u))$. 
One can check the defintions and observe that perfect matchings in $H_u$ correspond to the available permutations $\rho\in S_q$ for $u$ in the sense that setting $\om_u$ to an available $\rho$ yields a valid list packing. 
The heat-bath Glauber dynamics for $L$-packings thus works in much the same way as for $L$-colorings. 
The transition from a state $\om$ is defined by choosing a vertex $u\in V(G)$ uniformly at random,  and then a perfect matching in $H_u(G,L,\om)$ uniformly at random. It is straightforward to check that the chain is reversible and has uniform stationary distribution; we prove that it is ergodic (for large enough $q$) in Lemma~\ref{l:ergodic}.

An important consideration when $G$ has maximum degree $\De$ is that $H_u$ has minimum degree $q-\De$. This follows from the properties of a list packing: at packing index $i$ any color that is not available must be used by $f_i$ on a neighbor of $u$ and there are at most $\De$ such neighbors. Similarly, for a color $c_j$ with color index $j$, any  packing index $i$ in which $c_j$ is not available is explained by $c_j$ being used by $f_i$ on $N(u)$. 
Since the colorings in a packing are pairwise-disjoint, each such index must be due to distinct neighbors of $u$, of which there are at most $\De$. 
It is useful to observe that for any $\omega\in S_q^n$, even one that may not be a proper list packing in the sense that the $q$ list colorings it represents may not be proper, the definition of availability graph still makes sense and the observation on the minimum degree still applies. 
That is, the key property of $q$-list packings on graphs of maximum degree $\De$ that yields the minimum degree bound $q-\De$ is that the list colorings represented are pairwise-disjoint.
This fact is convenient in the proof of Lemma~\ref{l:couple-unif-pms}.

To construct the coupling we consider the weighted, directed graph $\Ga$ on $\Om$ where $(\om,\om')$ is an edge if and only if the list packings $\om$ and $\om'$ differ at exactly one vertex. For the edge $(\om,\om')\in E(\Ga)$, let $v$ be the unique vertex at which the two packings differ, and assign weight $\dc(\om_v,\om'_v)$ to the edge.
As with list coloring, the key computation is how much the expected distance changes for one step of the coupling in the case that we start at $(\om,\om')$ and update a neighbor $u$ of the unique vertex $v$ at which $\om$ and $\om'$ differ.
We prove the following result which plays the role of Lemma~\ref{l:couple-uni-elts} in our proof.

\begin{restatable}{lem}{lcoupleunifpms}\label{l:couple-unif-pms}
    There is a universal constant $C$ such that if $q>C\De^2$ the following holds.

    Let $(\om,\om')\in E(\Ga)$ be an edge of weight $\psi$ in $\Ga$ and let $v$ be the unique vertex at which $\om$ and $\om'$ differ. 
    Let $u\in N(v)$ and consider the availability graphs $H$ and $H'$ for $u$ in the packings $\om$ and $\om'$ respectively. 

    Then there exists a coupling $\ga$ of the uniform distributions on perfect matchings on $H$ and $H'$ which satisfies
     \[
        \E_{(\rh,\rh')\sim \ga} [\dc(\rh,\rh')] \le\fc{\psi}{2\De}.
     \]
\end{restatable}

\noindent
Note that in the conclusion any bound strictly better than $\psi/\De$ suffices for an application of Theorem~\ref{t:pathcoupling}, and one might expect to obtain results in the case of exactly $\psi/\De$ as well~\cite{Jer95,BD97,DG99}. 
We do not attempt to optimize the constant $C$ in our argument and hence $\psi/(2\Delta)$ is sufficient.

Lemma~\ref{l:couple-unif-pms} is a simple corollary of Lemma~\ref{l:couple-dist}, and with these results in hand the rest of the argument for Theorem~\ref{t:main-count} is standard.

\subsection{Organization}

In Section~\ref{s:dynamics} we define a Markov chain on list packings, prove ergodicity and establish rapid mixing given our results on coupling perfect matchings.
In Section~\ref{s:couple-matchings} we prove Lemmas~\ref{l:couple-dist} and~\ref{l:couple-unif-pms}.
We conclude with some remarks in Section~\ref{s:rem}.

\section{Glauber dynamics for list packing}\label{s:dynamics}

In this section we fix an $n$-vertex graph $G$ of maximum degree $\De$, a $q$-list assignment $L$ of $G$, and let $\Om$ be the set of $L$-packings of $G$.

We consider heat-bath Glauber dynamics $\cM=\cM(G,L)$ for list packing. Given a state $\om\in\Om$ and a vertex $u\in V(G)$, we say that a permutation $\rh$ in $S_q$ is \emph{available} for $u$ in $\om$ if setting $\om_u$ to $\rh$ yields a valid list packing. 
Transitions of $\cM$ are defined as follows. 
From state $\om\in\Om$, choose vertex $u\in V(G)$ uniformly at random, choose an available permutation $\rh\in S_q$ uniformly at random, and let the new state be $\si$ given by $\si_u=\rh$ and $\si_v=\om_v$ for $v\ne u$. 
It is straightforward to check that this is a reversible Markov chain on $\Om$ with uniform stationary distribution.
The question of ergodicity is less straightforward, though the standard argument for list coloring adapts easily.
We require a simple corollary of Hall's classic result on perfect matchings in bipartite graphs.

\begin{lem}[Corollary of Hall's theorem~\cite{PHal35}]\label{l:hall}
    Let $H$ be a bipartite graph with $q$ vertices on each side and minimum degree $d\ge q/2$. Then $H$ contains a perfect matching.
\end{lem}
\begin{proof}[Sketch proof]
Let the bipartition be $A\sqcup B$ and consider a subset $X\subset A$. If $1\le |X|\le d$ then $|N(X)|\ge d\ge |X|$. If instead $d+1\le |X|\le q$ then $N(X)$ must be all of $B$ since for each neighborhood of a vertex in $B$ must intersect $X$. The result follows from Hall's theorem.
\end{proof}

\begin{lem}\label{l:ergodic}
    For $q\ge2\Delta+2$ the Markov chain $\cM$ is ergodic.
\end{lem}
\begin{proof}
Because there are self-transitions at every state, 
    it suffices to show that the (finite) state space is connected. 

    Every state $\om\in\Om$ can be connected to an arbitrary $\om'\in\Om$ as follows. 
    Label the vertices of $G$ with the integers $1,\dotsc, n$ arbitrarily, and for $i=1,\dotsc,n$, sequentially turn $\om_i$ into $\om_i'$ as follows.
    If setting $\om_i$ equal to $\om'_i$ is not possible (i.e.\ $\om'_i$ is not available for $i$ in $\om$) then it's because some neighbor $j$ of $i$ with $j>i$ uses a color at a particular packing index which conflicts with $\om_i$. 
    To solution is to repack each such neighbor $j$ in $\om$, i.e.\ change $\om_j$ to a new permutation of $L(j)$, in turn such that the repacking is proper \emph{and} avoids any such conflicts. 
    These repacking steps are steps of the chain.

    To perform the repacking, we find a perfect matching in a suitable modification of the availability graph $H_j'=H_j(G,L,\om')$. 
    Recall that perfect matchings in $H_j'$ correspond to the available permutations $\rho$ for $u$ in $\om'$. 
    We have an additional condition on the permutation $\rho$ that we seek, namely that $\rho$ does not correspond to a color choice for $j$ that is incompatible with turning $\om_i$ into $\om_i'$.
    We can encode this in the availability graph by deleting an edge $(a,b)$ such that the $b$-th color in $L(j)$ is used at packing  index $a$ in color choices represented by $\om'_i$. 
    After this modification, the availability graph has minimum degree $q-\Delta-1$, so the condition $q\ge 2\Delta+2$ allows for an application of Lemma~\ref{l:hall}. 
    This shows that the necessary repackings exist, and thus that $\Omega$ is connected by transitions of the chain that occur with positive probability.
\end{proof}

\begin{thm}\label{t:mixingtime}
    The mixing time $\tmix(\ep)$ of $\cM$ is at most $O(n\log(n/\ep))$.
\end{thm}
\begin{proof}
    We apply Theorem~\ref{t:pathcoupling} with the following coupling defined on edges of the weighted graph $\Ga$ on $\Om$ such that $(\om,\om')$ is an edge of weight $\dc(\om_v,\om_v')$ whenever $\om$ and $\om'$ differ at a single vertex $v$.
    Lemma~\ref{l:ergodic} shows that $\Ga$ is connected, and we note that the diameter $D$ is at most $(\Delta+1)(q-1)n=O(n)$.
    This is because the sequence of steps constructed in Lemma~\ref{l:ergodic} consists of edges of $\Ga$, the total number of steps is at most $(\Delta+1)n$ because we repack each vertex $v$ at most once for each neighbor of $v$ to avoid conflict and at most once more to agree with $\om'$. As the Cayley distance on $S_q$ takes values in $\{0,1,\dotsc,q-1\}$, the diameter bound follows. 
    
    Let $(\om,\om')$ be an edge of $\Ga$ of weight $\psi$. We define the coupling as follows. We choose $u\in V(G)$ uniformly at random and update $u$ in both packings. If $u\notin N(v)$ then the sets of available permutations of $L(u)$ in both packings are identical and we choose one uniformly at random to use in both chains. If $u=v$ this results in a distance of zero, else the distance is unchanged. 
    If $u\in N(v)$ then we use the coupling of Lemma~\ref{l:couple-unif-pms} to choose the permutations of $L(u)$ in the packings.
    Let $(\si,\si')$ be the random state after one step of the coupling started from $(\om,\om')$. Then
\[ \E[\de(\si,\si')] \le \psi + \frac{1}{n}\left(-\psi + \De\frac{\psi}{2\De}\right) = \psi\left(1-\rc{2n}\right).\]
Theorem~\ref{t:pathcoupling} now gives mixing time $\tmix(\ep)=O(n\log(n/\ep))$.
\end{proof}

The proof of Theorem~\ref{t:main-count} using Theorem~\ref{t:mixingtime} is now entirely standard. 
We sketch the argument here. 
Counting list packings is a \emph{self-reducible} problem in the sense of~\cite{JVV86} and so having an almost-uniform sampler, which follows from running the Markov chain for polynomially many steps, is equivalent to having an FPRAS. 
Concretely, construct a sequence $G=G_m\supset \dotsb \supset G_1 \supset G_0 = (V(G),\emptyset)$ of graphs by starting from $G$ and removing an arbitrary edge to form the next member of the sequence. 
For a fixed $q$-list assignment $L$ of $G$ we can let $\Om_i$ be the set of $L$-packings of $G_i$ and write
\[ |\Om_m| = |\Om_0|\prod_{i=1}^{m}\fc{|\Om_{i}|}{|\Om_{i-1}|}. \]
We have $|\Om_0|=(q!)^n$ and can estimate each ratio in the product as it's the probability that when we choose a unformly random $\om\in \Om_{i-1}$ we have $\om\in\Om_i$. 
We also note that $\fc{1}{1+q!} \le \fc{|\Om_{i}|}{|\Om_{i-1}|}\le 1$
because $\Om_i\subset \Om_{i-1}$ and we can construct the following bipartite graph $B$ on $(\Om_{i-1}\setminus \Om_i)\sqcup \Om_i$. Let $uv$ be the edge removed from $G_i$ to form $G_{i-1}$, and include the edge $(\om,\om')$ in $B$ if $\om'$ can be obtained from $\Om$ by permuting the list of $u$. Then any $\om\in\Om_{i-1}\setminus \Om_i$ is connected to at least one element in $\Om_i$ by Lemma~\ref{l:hall}, and there are at most $q!$ ways to permute $L(u)$ so from the other side the degrees are at most $q!$. 
We do not attempt to optimize this argument; more intricate arguments yield stronger lower bounds, but we are merely interested in a bound independent of $n$.
This observation lets us repeat the analysis of Jerrum~\cite{Jer95} for the case of colorings in the setting of list packings. 
Briefly, we use multiple copies of the almost uniform sampler offered by the Markov chain to estimate each ratio. 
To make this work, one has to bound the variance of the estimator for each ratio and manage the overall error with Chebyshev's inequality, but this is standard. See e.g.,~\cite{Jer95}.

\section{Coupling matchings}\label{s:couple-matchings}

In this section we prove Lemmas~\ref{l:couple-dist} and~\ref{l:couple-unif-pms}.
We first collect some results on auxiliary bipartite graphs.

\begin{lem}\label{l:couple}
    Let $G=(V,E)$ be a bipartite graph with bipartition $V=L\cup R$. Let $\mu_L$ and $\mu_R$ be probability distributions over $L$ and $R$.
    Let 
    \[
        p = \min_{A\subeq L} [1 - \mu_L(A) + \mu_R(N(A))].
    \]
    There exists a coupling $\ga$ of $\mu_L$ and $\mu_R$ such that when $(v,w)\sim\ga$, with probability $p$ $vw$ is an edge of $G$.
\end{lem}
\begin{proof}
This follows from the max-flow min-cut theorem after introducing a source vertex $s$, a sink vertex $t$, and introducing the following edges:
\begin{enumerate}
    \item for $v\in L$, the edge $sv$ with capacity $\mu_L(v)$,
    \item for $v\in R$, the edge $vt$ with capacity $\mu_R(v)$,
    \item for each edge $e\in G$, an edge with capacity $\iy$.
\end{enumerate}
A max flow corresponds to a coupling $\ga$ of the type we require with maximum probability that $vw$ is an edge. 
The value of the max flow is the same as the value of the min cut. 
If $S,T$ is a finite cut with $s\in S$, $t\in T$, letting $A=S\cap L$, $B = S\cap R$, we must have $N(A)\subeq B$, and the value is $\mu_L(A^c)+\mu_R(B)\ge \mu_L(A^c)+\mu_R(N(A))$. Equality is achieved for $B=N(A)$. Taking the minimum over $A$ then gives the lemma.
\end{proof}

\begin{cor}\label{c:couple}
    Let $G=(V,E)$ be a bipartite graph with bipartition $V=L\cup R$. Suppose that each vertex in $L$ has degree contained in $[m_L,M_L]$ and each vertex in $R$ has degree contained in $[m_R,M_R]$. Then there exists a coupling $\ga$ of the uniform distributions $\mathcal{U}_L$ and $\mathcal{U}_R$ on $L$ and $R$ respectively such that when $(v,w)\sim\ga$, with probability at least $\fc{m_Lm_R}{M_LM_R}$, $vw$ is an edge of $G$.
\end{cor}
\begin{proof}
    For each subset $A\subeq L$, we can bound the number of edges between $A$ and $N(A)$:
    \[
    m_L |A|\le 
    |E(A,N(A))| \le M_R |N(A)|.
    \]
    We also have
    \[
    m_R |R| \le |E| \le M_L |L|.
    \]
    Then for any set $A\subeq L$, 
    \[
\fc{|N(A)|/|R|}{|A|/|L|} = \fc{|N(A)|}{|A|}
\cdot \fc{|L|}{|R|} \ge \fc{m_Lm_R}{M_LM_R}.
    \]
    Then for any such $A$,
    \[
1-\mathcal{U}_L(A) + \mathcal{U}_R(N(A)) 
\ge 1- \mathcal{U}_L(A) + \fc{m_Lm_R}{M_LM_R}\cdot \mathcal{U}_L(A)
\ge \fc{m_Lm_R}{M_LM_R}.
    \]
    The conclusion follows from Lemma \ref{l:couple}.
\end{proof}

We are now ready to prove Lemma~\ref{l:couple-dist}, which we restate for convenience.
\lcoupledist*

\begin{proof}%
Label the bipartitions by $[q]$. 
    Without loss of generality the edge $e$ is $(1,1)$. 
    We define an auxiliary bipartite graph on $\mathcal L\cup \mathcal R$ as follows. 
    For $\rh\in \mathcal L$ and $\rh'\in \mathcal R$, connect $\rh$ and $\rh'$ by an edge if $\rh (1\,i\,j) = \rh'$ for some $i,j\in [q]\bs\{1\}$ with $i\ne j$, i.e., they differ by a 3-cycle containing 1.

    We bound the degrees of arbitrary matchings $\rh\in \mathcal L$ and $\rh'\in \mathcal R$. 
    Given $\rh$, we count the number of $i,j$ for which $\rh(1\,i\,j)\in \mathcal R$. 
    Let $N(i) = \set{j}{(i,j)\in E}$ and $N'(j) = \set{i}{(i,j)\in E}$. 
    Given $\rh\in L$, we know that $\rh(1)=1$.
    We choose $j \in N'(1) \bs \{1\}$; there are at least $q-\De -1$ choices. 
    Any possible value of $i$ must be in the set $S_-\bs \{j\}$, where 
    $S_-:= \set{\rh^{-1}(k)}{k\in N(1)\bs \{1\}}$; $S_-\bs \{j\}$ has size at least $q-\De-2$. 
    A valid pair $(i,j)$ is exactly one where $j\in N'(1)\bs \{1\}$, $i\in S_-\bs \{j\}$, and $(i,\rh(j))\in E$.
    Since at most $\De+1$ of edges $(i,\rh(j))$, $j\in N'(1)\bs \{1\}$ can land outside $E$, at least $q-2\De-3$ of these edges are valid, i.e., there are at least $(q-\De-1)(q-2\De -3)$ valid choices of $(i,j)$. There are at most $(q-1)(q-2)$ choices.

    Next, given $\rh'\in\mathcal R$, we count the number of $i,j\in [q]\bs \{1\}$, $i\ne j$ for which $\rh'=\rh(1\,i\,j)$ where $\rh\in\mathcal L$. First, note we must have $\rh'(j)=1$. 
    The requirement on $i$ is that $i\in N'(\rh'(1))$ and $\rh'(i)\in N(j)$. There are at least $q-\De-2$ indices besides $1$ and $j$ satisfying each condition, so at least $q-2\De-4$ possible indices. There are at most $q-2$ choices. By Corollary~\ref{c:couple}, there is a coupling $\ga$ of $\mathcal{U}_{\mathcal L}$ and $\mathcal{U}_{\mathcal R}$ such that with probability at least
    \[
p:=\fc{(q-\De -1)(q-2\De-3)}{(q-1)(q-2)} \cdot \fc{q-2\De-4}{q-2},
    \]
    $(\rh,\rh')\in E$ are connected by an edge and hence have Cayley distance at most two. 
    For appropriate choices of $C_1,K>0$, $1-p\le {K\De}/{q}$ for $q\ge C_1\De$. 
    For this coupling $\ga$, we hence have
    \[
    \E_{(\rh,\rh')\sim \ga} [\dc(\rh,\rh')]
    \le \fc{K\De}{q}\cdot q + \pa{1-\fc{K\De}{q}} \cdot 2\le K'\De
    \]
    for an appropriate constant $K'>0$.
    
    Next, note that we can define a coupling between $\mathcal{U}_{\mathcal L\cup \mathcal R}$ and $\mathcal{U}_{\mathcal R}$ as a mixture of the identity coupling between $\mathcal{U}_{\mathcal R}$ and $\mathcal{U}_{\mathcal R}$ (with probability $\fc{|\mathcal R|}{|\mathcal L|+|\mathcal R|}$) and the above coupling (with probability $\fc{| \mathcal L|}{|\mathcal L|+|\mathcal R|}$). The expected distance for this coupling is then 
    \[
    \fc{|\mathcal L|}{|\mathcal L|+|\mathcal R|}\cdot 
    K'\De 
    \le \fc{q-2}{(q-2)+(q-\De-1)(q-2\De-3)} \cdot 
    K'\De \le \fc{C_2\De}{q}
    \]
    for appropriate $C_2>0$, as needed.
\end{proof}

\lcoupleunifpms*

\begin{proof}%

    Formally, we proceed by induction on $\psi$, though we need a more general statement for the induction hypothesis. It is convenient to relax the requirement that $\om$ and $\om'$ are valid list packings. 
    For the application, it is important that the pair $(\om,\om')$ are valid list packings which agree on every vertex except $u$, but we construct the coupling of perfect matchings in $H$ and $H'$ using a sequence of \emph{near-valid} list packings $\om''\in S_q^n$ in the sense that $\om''$ agrees with $\om$ and $\om'$ on all vertices except $u$, but we allow the (pairwise-disjoint) colorings it represents to have monochromatic edges incident to $v$. 
    We can still construct availability graphs for $u$ in these near-valid packings and consider their sets of perfect matchings for the purposes of constructing an eventual coupling of the perfect matchings in $H$ and $H'$.
    With these definitions in place, the generalization that we prove by induction is the statement obtained by replacing the assumption that $\om$ and $\om'$ are valid packings with the assumption that they are near-valid.
    
    The base case is $\psi=0$ in which the trivial coupling suffices as $H$ and $H'$, and hence their sets of perfect matchings, are identical. 

    The induction step follows from Lemma~\ref{l:couple-dist}. 
    The fact that $\dc(\om_v,\om'_v)=\psi$ means that there is a sequence of transpositions $\tau_1,\dotsc,\tau_\psi$ such that $\om_v=\tau_\psi\dotsb \tau_1\om'_v$.
    Let $H''$ be the availability graph of the vertex $u$ in the near-valid packing $\om''$ such that $\om''_v = \tau_\psi\om_v$ and $\om''$ agrees with $\om$ on all other vertices.    
    By induction, there is a coupling $\ga'$ of the uniform distributions on perfect matchings in $H''$ and $H'$ such that 
    \[ \E_{(\rh'',\rh')\sim\ga'}[\dc(\rh'',\rh')] \le \fc{\psi-1}{2\De}. \]
    
    Without loss of generality, suppose that $\om_v$ is the identity. 
    Let $\tau_\psi=(i\,j)$, and note that this gives 
    \begin{align*}
    E(H)\setminus E(H'') &\subset \{(i,j), (j,i)\}&
    &\text{and}&
    E(H'')\setminus E(H) &\subset \{(i,i), (j,j)\}.&
    \end{align*}
    This is because any difference between the edges of $H$ and $H'$ is explained by applying $\tau$ to $\om_v$. When $\tau$ is a transposition, we swap the packing index of the coloring at which two colors in the list of $v$ are used, which can swap two edges of the complement of the availability graph. 
    It can be the case that the swapped color indices refer to different colors in $L(u)$ which is why we do not have equality.
    
    We start with the case that 
    \begin{align*}
    E(H)\setminus E(H'') &= \{(i,j), (j,i)\}&
    &\text{and}&
    E(H'')\setminus E(H) &= \{(i,i), (j,j)\},&
    \end{align*}
    the other cases are similar.
    Let $X$ be the set of perfect matchings in $H$ and let $X''$ be the set of perfect matchings in $H''$. We seek a coupling of $\mathcal{U}_{X}$ and $\mathcal{U}_{X''}$ which we will combine with the coupling $\ga'$ to obtain the desired result. 
    
    We apply Lemma~\ref{l:couple-dist} to $H$ and $H-(i,j)$, yielding a coupling $\ga_1$ of $\mathcal{U}_{X}$ and $\mathcal{U}_{Y}$, where $Y$ is the set of perfect matchings in $H-(i,j)$.
    We can apply Lemma~\ref{l:couple-dist} again to $H-(i,j)$ and $H-(i,j)-(j,i)$, yielding a coupling $\ga_2$ of $\mathcal{U}_{Y}$ and $\mathcal{U}_{Z}$, where $Z$ is the set of perfect matchings in $H-(i,j)-(j,i)$.
    There is a slight technicality here as the minimum degree of $H-(i,j)$ is $q-\Delta-1$, but this can be handled by setting $\Delta$ to $\Delta+1$ and adjusting the constants slightly.
    Analogously, starting from $H''$, we construct a coupling $\ga_1''$ of $\mathcal{U}_{X''}$ and $\mathcal{U}_{Y''}$ and a coupling $\ga_2''$ of $\mathcal{U}_{Y''}$ and $\mathcal{U}_{Z}$, where $Y''$ is the set of perfect matchings in $H''-(i,i)$. 
    A careful composition of these couplings gives the result. 
    The composition of these couplings yields a distribution on $X\times Y\times Z\times Y''\times X''$ which is uniform on each individual set in the Cartesian product, and such that the expected distance between permutations from adjacent sets in the Cartesian product is at most $C_2\Delta/q$. Taking the first and last coordinate yields a coupling $\ga''$ of $\mathcal{U}_X$ and $\mathcal{U}_{X''}$ such that
    \[
    \E_{(\pi,\pi'')\sim \ga''} [\dc(\pi,\pi'')] <\fc{4C_2\Delta}{q}.
    \]
    This can be combined with $\ga'$ obtained by induction in the same way. Simple composition yields a distribution on $X\times X''\times X'$ which is uniform on each individual set in the Cartesian product. 
    Taking the first and last coordinates we have a coupling $\ga$ of $\mathcal{U}_X$ and $\mathcal{U}_{X'}$ such that 
    \[
    \E_{(\pi,\pi')\sim \ga} [\dc(\pi,\pi')] < \fc{\psi-1}{2\De} + \fc{4C_2\De}{q}.
    \]
    Since we assume $q\ge C\Delta^2$, for a large enough $C$ this is at most $\psi/(2\Delta)$ as required.
    The other cases proceed similarly, but require fewer applications of Lemma~\ref{l:couple-dist} and yield a stronger upper bound.
\end{proof}

\section{Concluding remarks}\label{s:rem}

Many natural questions remain unanswered. We have chosen to extend some of the simplest and earliest techniques for list coloring to list packing, but there are many more recent improvements to consider. 
Dyer and Greenhill~\cite{DG98} study a Markov chain on (list) colorings whose transitions are defined by properly recoloring both endpoints of a uniform random edge and show that its mixing time is less than that of Glauber dynamics studied in~\cite{Jer95,SS97}.
The flip dynamics employed by Vigoda~\cite{Vig99} for counting colorings is an important technique and while one can consider analogous dynamics for list packings, the analysis is potentially formidable. 
Further, the use of more advanced Markov chain techniques to give perfect sampling could be interesting in the setting of list packing.

We finish with a natural conjecture on approximately counting list packings.

\begin{conj}
    For each $\De$ and $q\ge 2\Delta$ there is an FPRAS for counting the number of $q$-list packings of graphs of maximum degree $\De$.
\end{conj}

\noindent
At the time of writing, we know of no reason that the lower bound on $q$ cannot be reduced to, say, $\Delta+1$. The value $2\Delta$ represents a significant barrier in the sense that the existence of a list packing when $q\ge 2\Delta$ is elementary (though arguably not entirely trivial).

\section{Acknowledgement}

Part of this work was completed at Random Theory 2023.

\printbibliography

\end{document}